\documentclass[12pt, reqno]{amsart}
\usepackage{amsmath, amsthm, amscd, amsfonts, amssymb, graphicx, color}
\usepackage[bookmarksnumbered, colorlinks, plainpages]{hyperref}
\hypersetup{colorlinks=true,linkcolor=red, anchorcolor=green, citecolor=cyan, urlcolor=red, filecolor=magenta, pdftoolbar=true}

\newtheorem{theorem}{Theorem}[section]
\newtheorem{lemma}[theorem]{Lemma}
\newtheorem{proposition}[theorem]{Proposition}
\newtheorem{corollary}[theorem]{Corollary}
\theoremstyle{definition}
\newtheorem{definition}[theorem]{Definition}
\newtheorem{example}[theorem]{Example}

\theoremstyle{remark}

\numberwithin{equation}{section}

\begin{document}

\setcounter{page}{1}

\title[Hyperrigid generators in $C^*$-algebras]{Hyperrigid generators in $C^*$-algebras}

\author[P. Shankar]{P. Shankar}

\address{Indian Statistical Institute, Statistics and Mathematics Unit, 8th Mile, Mysore Road,
Bangalore, 560059, India.}
\email{shankarsupy@gmail.com, shankar\_vs@isibang.ac.in}



\date{\today}

\begin{abstract}
In this article, we show that, if $S\in \mathcal{B}(H)$ is irreducible and essential unitary, then $\{S,SS^*\}$ is a hyperrigid generator for the unital $C^*$-algebra $\mathcal{T}$ generated by $\{S,SS^*\}$. We prove that, if $T$ is an operator in $\mathcal{B}(H)$ that generates an unital $C^*$-algebra $\mathcal{A}$ then $\{T,T^*T,TT^*\}$ is a hyperrigid generator for $\mathcal{A}$. As a corollary it follows that, if $T\in \mathcal{B}(H)$ is normal then $\{T,TT^*\}$ is hyperrigid generator for the unital $C^*$-algebra generated by $T$ and if $T\in \mathcal{B}(H)$ is unitary then $\{T\}$ is hyperrigid generator for the $C^*$-algebra generated by $T$. We show  that if $V\in \mathcal{B}(H)$ is an isometry (not unitary) that generates the $C^*$-algebra $\mathcal{A}$ then the minimal generating set $\{V\}$ is not hyperrigid for $\mathcal{A}$. 

\end{abstract} \maketitle

\section{Introduction}
The classical theorems of Korovkin impressed several mathematicians since their discovery for the simplicity and the potential. Positive approximation process play a fundamental role in the approximation theory and it appears in a very natural way in several problems dealing with the approximation of continuous functions and qualitative properties such as monotonicity, convexity, shape preservation and so on. 

Korovkin \cite{PPK60} made a assertion that, if a sequence of positive linear maps\linebreak $\phi_n:C[0,1]\rightarrow C[0,1]$, $n=1,2,3,...,$ has the property 
$$\lim\limits_{n\rightarrow \infty}||\phi_n(f_k)-f_k||=0, ~~k=0,1,2,$$
for the three functions $f_0(x)=1, f_1(x)=x, f_2(x)=x^2$ then 
$$\lim\limits_{n\rightarrow \infty}||\phi_n(f)-f||=0,~~\forall ~~ f\in C[0,1].$$
The set $\{1,x,x^2\}$ is called a \textit{Korovkin set} or \textit{test set}. Korovkin \cite{PPK60} showed that, the set $\{1,x\}$ is not a Korovkin set. Therefore, the set $\{1,x,x^2\}$ is a minimal set to satisfy the above assertion.

Korovkin's theorem generated considerable activity among researchers in approximation theory. The generalizations make essential use of the \textit{Choquet boundary} in one way or another. Saskin \cite{YAS66} proved a remarkable theorem. Let $G$ be a subset of $C(X)$ that separates points of compact Hausdorff space $X$ and contains constant function $1$. Then $G$ is a Korovkin set in $C(X)$ if and only if the Choquet boundary $\partial G$ of $G$ is $X$. That is $\partial G=X$.

Arveson \cite{WBA11} initiated the study of noncommutative approximation theory focusing on the question: How does one determine whether a set of generators of a $C^*$-algebra is \textit{hyperrigid}? Arveson \cite{WBA11} introduced a noncommutative counterpart of Korovkin set as follows:

\begin{definition}
A finite or countably infinite set $\mathcal{G}$ of generators of a $C^*$-algebra $\mathcal{A}$ is said to be \textit{hyperrigid} if for every faithful representation $\mathcal{A}\subseteq \mathcal{B}(H)$ of $\mathcal{A}$ on a Hilbert space $H$ and every sequence of unital completely positive (UCP) maps $\phi_n:\mathcal{B}(H)\rightarrow \mathcal{B}(H)$, $n=1,2,...,$
$$\lim_{n\rightarrow \infty}||\phi_n(g)-g||=0,~\forall~ g\in \mathcal{G} \Longrightarrow \lim_{n\rightarrow \infty}||\phi_n(a)-a||=0,~\forall~ a\in \mathcal{A}. $$
\end{definition}

Note that, a set $\mathcal{G}$ is hyperrigid if and only if $\mathcal{G}\cup \mathcal{G}^*$ is hyperrigid if and only if the linear span of $\mathcal{G}$ is hyperrigid. If $\mathcal{A}$ is unital, then $\mathcal{G}$ is hyperrigid if and only if $\mathcal{G}\cup \{1\}$ is hyperrigid \cite[Proposition 2.1]{GS17}.

The following characterization of hyperrigid operator systems due to Arveson \cite{WBA11} is more of a workable definition of hyperrigidity of  operator systems.

\begin{theorem}\cite[Theorem 2.1]{WBA11}\label{hypuep}
Let $\mathcal{S}$ be a separable operator system generating the $C^*$-algebra $\mathcal{A}=C^*(\mathcal{S})$ then $\mathcal{S}$ is hyperrigid if and only if every nondegenerate representation $\pi:\mathcal{A}\rightarrow \mathcal{B}(H)$  on a separable Hilbert space, $\pi{|_\mathcal{S}}$ has the unique extension property in the sense that the only unital completely positive (UCP) map $\phi:\mathcal{A}\rightarrow \mathcal{B}(H)$ that satisfies $\phi{|_\mathcal{S}}=\pi{|_\mathcal{S}}$ is $\phi=\pi$ itself.
\end{theorem}

The interesting examples of hyperrigid generators are obtained by a direct application of the above criterion. Arveson \cite{WBA11} established the noncommutative strengthening of a classical approximation-theoretic result of Korovkin.

\begin{theorem}\cite[Theorem 3.1]{WBA11}
Let $X\in B(H)$ be a self adjoint operator with atleast 3 points in its spectrum and let $\mathcal{A}$ be the $C^*$-algebra generated by $X$. Then
\begin{itemize}
\item[(i)] $\mathcal{G}=\{X,X^2\}$ is a hyperrigid generator for $\mathcal{A}$, while

\item[(ii)] $\mathcal{G}_0=\{X\}$ is not hyperrigid generator for $\mathcal{A}$.
\end{itemize}
\end{theorem}

\begin{theorem}\cite[Theorem 3.3]{WBA11}\label{hypiso}
Let $V\in \mathcal{B}(H)$ be an isometry that generates a $C^*$-algebra $\mathcal{A}$. Then $\mathcal{G}=\{V, VV^*\}$ is hyperrigid generator for $\mathcal{A}$.
\end{theorem}

Arveson \cite{WBA11} essentially used the noncommutative Choquet boundary. He found the hyperrigid generator for compact operators $\mathcal{K}(H)$. 

\begin{theorem}\cite[Theorem 8.1]{WBA11}
Let $V\in \mathcal{B}(H)$ be an irreducible compact operator with cartesian decomposition $V=A+iB$, where $A$ is a finite rank positive operator and $B$ is essential with Ker$B=\{0\}$. Then
\begin{itemize}
\item[(i)] $\mathcal{G}=\{V,V^2\}$ is hyperrigid generator for $C^*$-algebra $\mathcal{K}(H)$ of compact operators. In particular every sequence of unital completely positive maps $\phi_n:\mathcal{B}(H)\rightarrow \mathcal{B}(H)$ for which 
$$\lim_{n\rightarrow \infty} ||\phi_n(V)-V||=\lim_{n\rightarrow \infty}||\phi_n(V^2)-V^2||=0, $$
one has 
$$\lim_{n\rightarrow \infty}||\phi_n(K)-K||=0 $$
for every compact operator $K\in \mathcal{B}(H)$.

\item[(ii)] The smaller generating set $\mathcal{G}_0=\{V\}$ of $\mathcal{K}(H)$ is not hyperrigid.
\end{itemize}
\end{theorem}

Let $S=(S_1,...,S_d)$ denote the compression of the $d$-shift to the complement of a homogeneous ideal $I$ of $\mathbb{C}[z_1,...,z_d]$. Following the remark above, Kennedy and Shalit \cite[Theorem 4.12]{KS15} (see also \cite{KS16}) proved that, if homogeneous ideals are sufficiently non-trivial then $S$ is essentially normal if and only if it is hyperrigid as the generating set of a $C^*$-algebra.

The main purpose of this paper is to find the minimal hyperrigid generators for $C^*$-algebras. We show that, if $S\in \mathcal{B}(H)$ is irreducible and is an essential unitary and $\mathcal{G}=\{S,SS^*\}$. Let $\mathcal{T}=C^*(\mathcal{G})$ be the unital $C^*$-algebra generated by $\mathcal{G}$. Then $\mathcal{G}$ is a hyperrigid generator for $\mathcal{T}$. We prove that, if $T$ is an operator in $\mathcal{B}(H)$ that generates a unital $C^*$-algebra $\mathcal{A}$ and $\mathcal{G}=\{T,T^*T,TT^*\}$, then $\mathcal{G}$ is a hyperrigid generators for the unital $C^*$-algebra $\mathcal{A}$. As a corollary it follows that, if $T$ be a normal operator in $\mathcal{B}(H)$ that generate a unital $C^*$-algebra $\mathcal{A}$ and let $\mathcal{G}=\{T, TT^*\}$. Then $\mathcal{G}$ is hyperrigid generator for unital $C^*$-algebra $\mathcal{A}$. If $T$ be an unitary operator in $\mathcal{B}(H)$ that generate a  $C^*$-algebra $\mathcal{A}$ and let $\mathcal{G}=\{T\}$. Then $\mathcal{G}$ is hyperrigid generator for $C^*$-algebra $\mathcal{A}$. We show that, if $V\in \mathcal{B}(H)$ be an isometry (not unitary) that generates a $C^*$-algebra $\mathcal{A}$. Then the minimal generating set $\mathcal{G}_0=\{V\}$ is not hyperrigid generator for $C^*$-algebra $\mathcal{A}$.

\section{Preliminaries}
Here, we recall the necessary definitions, conventions and notations.

Let $H$ be a separable complex Hilbert space and let $\mathcal{B}(H)$ be the set of all bounded linear operators on $H$. A \textit{operator system} $\mathcal{S}$ in a $C^*$-algebra $\mathcal{A}$ is a self-adjoint linear subspace of $\mathcal{A}$ containing the identity of $\mathcal{A}$. A \textit{operator algebra} $\mathcal{A}_0$ in a $C^*$-algebra $\mathcal{A}$ is a unital subalgebra of $\mathcal{A}$.  Given a linear map $\phi$ from a $C^*$-algebra $\mathcal{A}$  into a $C^*$-algebra $\mathcal{B}$ we can define a family of maps $\phi_n:M_n(\mathcal{A})\rightarrow M_n(\mathcal{B})$ given by  $\phi_n([a_{ij}])=[\phi(a_{ij})]$, $n\in \mathbb{N}$. We say that $\phi$ is \textit{completely bounded} (CB) if $||\phi||_{\text{CB}}=\sup_{n\geq 1}||\phi_n||< \infty$. We say that $\phi$ is \textit{completely contractive} (CC) if $||\phi||_{\text{CB}} \leq 1$ and that $\phi$ is \textit{completely isometric} if $\phi_n$ is isometric for all $n\geq 1$. We say that $\phi$ is \textit{completely positive} (CP) if $\phi_n$ is positive for all $n\geq 1$, and that $\phi$ is \textit{unital completely positive} (UCP) if in addition $\phi(1)=1$.

\begin{definition}
Let $\mathcal{S}$ be an operator system that  generates a $C^*$-algebra $\mathcal{A}$. A unital completely positive map $\phi:\mathcal{S}\rightarrow \mathcal{B}(H)$ is said to have the \textit{unique extension property} if it has a unique extension to a UCP map $\widetilde{\phi}:\mathcal{A}\rightarrow \mathcal{B}(H)$
\end{definition}

The boundary representations of $\mathcal{A}$ for $\mathcal{S}$, which were introduced by Arveson \cite{WBA69}, are precisely the irreducible representations $\pi:\mathcal{A}\rightarrow \mathcal{B}(H)$ with the property that the restriction $\pi_{|_\mathcal{S}}$ has the unique extension property. The existence of boundary representations was an open question for some time. Arveson \cite{WBA08} proved the existence of boundary representions for separable $C^*$-algebras. Davidson and Kennedy \cite{DK15} settled the existence of boundary representations for general $C^*$-algebras. 

Arveson \cite{WBA11} tried to prove the non-commutative analogue of Saskin's theorem \cite{YAS66} using theory of noncommutative Choquet boundary for unital completely positive maps on $C^*$-algebras and noncommutative counterpart of the Korovkin's set which is the hyperrigid set. Arveson \cite{WBA11} proved that if the separable operator system is hyperrigid in the $C^*$-algebra then every irreducible representation of $C^*$-algebra is a boundary representation for the operator system. The converse to this result is called \textit{hyperrigidity conjecture}: that is, if every irreducible representation of a $C^*$-algebra is a boundary representation for a separable operator system then the operator system is hyperrigid.

Arveson \cite{WBA11} showed that the hyperrigidity conjecture is true for $C^*$-algebras with countable spectrum. Kleski \cite{CK14} established the hyperrigidity conjecture for all type-I $C^*$-algebras with additional assumptions on the co-domain. Davidson and Kennedy\cite{DK16} proved the conjecture for function systems. Clouatre \cite{RC18} established the hyperrigidity conjecture with assumption of unperforated. The hyperrigidity conjecture is still open for general $C^*$-algebras. Namboodiri, Pramod, Shankar and Vijayarajan \cite{NPSV18} approached the hyperrigidity conjecture with weaker notions. They got the partial answers.

\section{Essential Unitary and hyperrigidity}
Let $\mathcal{B}(H)$ be the algebra of bounded linear operators on a separable complex Hilbert space $H$ and $\mathcal{K}(H)$ ideal of compact operators on $H$.\linebreak Let $\pi:\mathcal{B}(H)\rightarrow \mathcal{B}(H)/\mathcal{K}(H)$ be the natural surjection onto the Calkin algebra $\mathcal{B}(H)/\mathcal{K}(H)$. The operator $T\in \mathcal{B}(H)$ is called  essentially normal if $\pi(T)$ is normal in the Clakin algebra, or equivalently, $T^*T-TT^*$ is compact. The operator $S\in \mathcal{B}(H)$ is called essentially unitary if $\pi(S)$ is unitary in the Clakin algebra, or equivalently, $I-S^*S$ and $I-SS^*$ are compact. The above definitions can be found in \cite{BDF73}.

Here, we will have the following assumptions to proceed. Let $S$ be a irreducible and essential unitary but not unitary operator in $\mathcal{B}(H)$ and let  $\mathcal{G}=\{S,SS^*\}$. Let $\mathcal{S}$ be a operator system generated by $\mathcal{G}$. Let $\mathcal{T}=C^*(\mathcal{G})$ be the unital $C^*$-algebra generated by $\mathcal{G}$. The unital $C^*$-algebra $\mathcal{T}$ contains the compact operators $\mathcal{K}(H)$.

A representation $\rho:\mathcal{T}\rightarrow \mathcal{B}(H)$ is said to be singular representation if it annihilates the compact operators $\mathcal{K}(H)$. 

\begin{lemma}\label{inva}
Let $\rho:\mathcal{T}\rightarrow \mathcal{B}(H)$ be a representation, and let $\pi:\mathcal{T}\rightarrow \mathcal{B}(K)$ be a representation such that $\pi{|_{\mathcal{S}}}$ is a dilation of  $\rho{|_{\mathcal{S}}}$. Then the subspace $H$ is coinvariant for $\pi(\mathcal{S})$.
\end{lemma}
\begin{proof}
With respect to the decomposition $K=H\oplus H^\perp$. By assumption we have 
$$\pi(S)=\left( {\begin{array}{cc}
   \rho(S) & X \\
   Y       & Z \\
  \end{array} } \right)$$ 
Note that $X=P_H\pi(S)|_{H^\perp}$. We must prove that $X=0$. By assumption, 
$$\left( {\begin{array}{cc}
   \rho(SS^*) & X_0 \\
   Y_0       & Z_0 \\
  \end{array} } \right)=\pi(SS^*)=\pi(S)\pi(S)^*
  =\left( {\begin{array}{cc}
   \rho(S) & X \\
   Y       & Z \\
  \end{array} } \right)
  \left( {\begin{array}{cc}
   \rho(S)^* & Y^* \\
   X^*     & Z^* \\
  \end{array} } \right). $$ 
We get,  
  
$$\rho(SS^*)=\rho(S)\rho(S)^*+XX^*$$
Therefore, $XX^*=0$, and hence $X=0$. 
\end{proof}

\begin{proposition}\label{singular}
Suppose that $S$ is irreducible and essential unitary and\linebreak $\mathcal{G}=\{S,SS^*\}$. Let $\mathcal{S}$ be a operator system generated by $\mathcal{G}$ and $\mathcal{T}=C^*(\mathcal{G})$. Let $\rho:\mathcal{T}\rightarrow \mathcal{B}(H)$ be a singular representation. Then the restriction $\rho{|_\mathcal{S}}$ has unique extension property. 
\end{proposition}

\begin{proof}
We will use the fact that a UCP map $\phi '$ has the unique extension property if and only if $\phi '$  is \textit{maximal}, meaning that every UCP map that dilates $\phi '$ contains as a direct summand \cite[Proposition 2.4]{WBA08}. 

Let $K$ be a Hilbert space properly containing $H$. Let $\pi:\mathcal{T}\rightarrow \mathcal{B}(K)$ be a representation such that the restriction $\pi|_{\mathcal{S}}$ is a dilation of $\rho{|_\mathcal{S}}$. To show that the restriction $\rho{|_\mathcal{S}}$ has unique extension property, it is enough to show that the dilation $\pi$ is trivial, that is, $\pi|_{\mathcal{S}}=\rho|_{\mathcal{S}}\oplus\phi$ for some UCP map $\phi$.  

Using the Lemma \ref{inva}, we can decompose $K=H\oplus H^{\perp}$ and write
$$\pi(S)=\left( {\begin{array}{cc}
   \rho(S) & 0 \\
   Y       & Z \\
  \end{array} } \right).$$ 

Since $\rho$ is singular, $\rho(S)$ is unitary,
so it cannot be dilated to a compression. Therefore the dilation $\pi$ must be trivial.
  
\end{proof}

\begin{proposition}\label{identity}
Suppose that $S$ is irreducible and essential unitary and\linebreak $\mathcal{G}=\{S,SS^*\}$. Let $\mathcal{S}$ be a operator system generated by $\mathcal{G}$ and $\mathcal{T}=C^*(\mathcal{G})$.  Then the identity representation of $\mathcal{T}$ is a boundary representation for $\mathcal{S}$.
\end{proposition}
\begin{proof}
Since $S$ is irreducible and essential unitary. The unital $C^*$-algebra generated by $\mathcal{G}$ contains the compact operators, that is, $\mathcal{K}(H)\subseteq \mathcal{T}=C^*(\mathcal{G})$. The operator system $\mathcal{S} \subset \mathcal{T}$ is irreducible and contains the identity operator. By our assumption, $0\neq \mathcal{K}=I-SS^*\in \mathcal{S}$ is a compact operator, we have $||\mathcal{K}-\mathcal{K}||<||\mathcal{K}||$. Therefore, the quotient map $q:\mathcal{B}(H)\rightarrow \mathcal{B}(H)/\mathcal{K}(H)$ is not completely isometric on $\mathcal{S}$. Hence by boundary theorem of Arveson \cite[Theorem 2.1.1]{WBA72}, identity representation of  $\mathcal{T}$ is a boundary representation for $\mathcal{S}$.  
\end{proof}

\begin{theorem}\label{hypessuni}
Let $S$ be an irreducible and essential unitary and $\mathcal{G}=\{S,SS^*\}$. Let $\mathcal{T}=C^*(\mathcal{G})$ be the unital $C^*$-algebra generated by $\mathcal{G}$. Then $\mathcal{G}$ is a hyperrigid generator for $\mathcal{T}$.
\end{theorem}
\begin{proof}
Let $\mathcal{S}$ be the operator system generated by $\mathcal{G}$. Note that $\mathcal{G}$ is hyperrigid if and only if $\mathcal{S}$ is hyperrigid. By Theorem \ref{hypuep}, it suffices to show that for every nondegenerate representation $\rho$ of $\mathcal{T}$,  $\rho{|_\mathcal{S}}$ has the unique extension property. 

The Proposition \ref{singular} implies that every singular nondegenerate representation $\pi$ of $\mathcal{T}$, $\pi{|_\mathcal{S}}$ has the unique extension property. By Proposition \ref{identity}, the restriction of the identity representation of $\mathcal{T}$ to $\mathcal{S}$ has the unique extension property. Since every nondegenerate representation of $\mathcal{T}$ splits as the direct sum of a multiple of the identity representation and a singular nondegenerate representation and by \cite[Proposition 4.4]{WBA11} the unique extension property passes to direct sums. Hence every nondegenerate representation of $\mathcal{T}$ restricted to $\mathcal{S}$ has the unique extension property.  
\end{proof}

\begin{example}
Let $H$ be a Hilbert space having an orthonormal basis\linebreak $\{e_n:n\geq 0\}$. The unilateral shift $S$ is defined by $Se_n=e_{n+1}$. The $C^*$-algebra $\mathcal{T}$ generated by $S$ is called the Toeplitz $C^*$-algebra. Observe that $I-S^*S$ and $I-SS^*$ are compact, therefore $S$ is essential unitary. Also, $S$ is irreducible. The Toeplitz $C^*$-algebra $\mathcal{T}$ contains the compact operators $\mathcal{K}(H)$. We know that the set $\{S,SS^*\}$ also generates the Toeplitz $C^*$-algebra $\mathcal{T}$. Hence, by Theorem \ref{hypessuni}, The set $\{S,SS^*\}$ is hyperrigid generator for Toeplitz $C^*$-algebra $\mathcal{T}$. 
\end{example}

\section{Hyperrigid generators}

The main purpose of this section is to find the hyperrigid generators for the $C^*$-algebras generated by a single operator.

\begin{theorem}\label{R1}
Let $T$ be an operator in $\mathcal{B}(H)$ that generate a unital $C^*$-algebra $\mathcal{A}$ and let $\mathcal{G}=\{T,T^*T,TT^*\}$. Then $\mathcal{G}$ is hyperrigid generators for  unital $C^*$-algebra $\mathcal{A}$.
\end{theorem}

\begin{proof}
Let $\mathcal{S}$ be the operator system generated by $\mathcal{G}$. By Theorem \ref{hypuep}, it suffices to show that for every nondegenerate representation $\pi$ of $\mathcal{A}$, $\pi{|_\mathcal{S}}$ has the unique extension property.

Let $\pi:\mathcal{A}\rightarrow \mathcal{B}(H)$ be a nondegenerate representation. Let $\phi:\mathcal{A}\rightarrow \mathcal{B}(H)$ be a UCP map satisfying $\phi(T)=\pi(T), \phi(T^*T)=\pi(T^*T)$ and $\phi(TT^*)=\pi(TT^*)$.  We have to show that $\phi=\pi$ on $\mathcal{A}$.

Using Stinespring theorem, we can express $\phi$ in the form
$$\phi(S)=V^*\sigma(S)V,~~\forall~~ S\in \mathcal{A}.$$
Where $\sigma$ is a representation of $A$ on a Hilbert space $K$, $V:H\rightarrow K$ is an isometry, and which is minimal in the sense that $\overline{\sigma(A)VH}=K$.

We first claim that $\sigma(T)V=V\pi(T)$, We have
$$V^*\sigma(T)^*VV^*\sigma(T)V=\phi(T)^*\phi(T)=\pi(T)^*\pi(T)=\pi(T^*T)$$

Hence,

$$
\begin{array}{rcl}
 V^*\sigma(T)^*(1-VV^*)\sigma(T)V &=& V^*\sigma(T)^*\sigma(T)V-V^*\sigma(T)^*VV^*\sigma(T)V\\
                                  &=& V^*\sigma(T^*T)V-\pi(T)^*\pi(T)\\
                                  &=& \pi(T^*T)-\pi(T^*T)=0.
 
\end{array}
$$

$\sigma(T)$ leaves $VH$ invariant. Therefore $\sigma(T)V=VV^*\sigma(T)V=V\phi(T)=V\pi(T)$.

$$
\begin{array}{rcl}
VV^*\sigma(T)(1_K-VV^*)\sigma(T)^*VV^*&=&VV^*\sigma(T)\sigma(T)^*VV^*\\
                                      &  &-VV^*\sigma(T)VV^*\sigma(T)^*VV^*\\
                                      &=&VV^*\sigma(TT^*)VV^*-V\pi(T)\pi(T)^*V^*\\
                                      &=&V\pi(TT^*)V^*-V\pi(TT^*)V^*=0.
\end{array}
$$

Hence $(1_K-VV^*)\sigma(T)^*VV^*=0$, we conclude that $VH$ is invariant under both $\sigma(T)$ and $\sigma(T)^*$. Since $\mathcal{A}$ is generated by $T$ it follows that $\sigma(\mathcal{A})VH\subseteq VH$. By minimality we must have $VH=K$, which implies that $V$ is unitary and therefore $\phi(S)=V^{-1}\sigma(S)V$ is a representation. Since $\phi$ agrees with $\pi$ on a generating set. Therefore $\phi=\pi$ on $\mathcal{A}$.
\end{proof}

\begin{corollary}
Let $T$ be a normal operator in $\mathcal{B}(H)$ that generate a unital $C^*$-algebra $\mathcal{A}$ and let $\mathcal{G}=\{T, TT^*\}$. Then $\mathcal{G}$ is hyperrigid generator for unital $C^*$-algebra $\mathcal{A}$.
\end{corollary}

\begin{corollary}
Let $T$ be an unitary operator in $\mathcal{B}(H)$ that generate a $C^*$-algebra $\mathcal{A}$ and let $\mathcal{G}=\{T\}$. Then $\mathcal{G}$ is hyperrigid generator for $C^*$-algebra $\mathcal{A}$.
\end{corollary}

\begin{proposition}
Let $V\in \mathcal{B}(H)$ be an isometry (not unitary) that generates a $C^*$-algebra $\mathcal{A}$. Then 
\begin{itemize}
\item[(i)] $\mathcal{G}=\{V,VV^*\}$ is hyperrigid generator for $\mathcal{A}$.
\item[(ii)] The smaller generating set $\mathcal{G}_0=\{V\}$ is not hyperrigid.
\end{itemize}
\end{proposition}

\begin{proof}
(i) follows from the Theorem \ref{hypiso}. Now we will prove (ii), let $\mathcal{S}$ be the operator system generated by $\mathcal{G}_0$. Let $Id$ denote the identity representation of a $C^*$-algebra $\mathcal{A}$. Let $V^*Id(\cdot)V$ be a completely positive map on the $C^*$-algebra $\mathcal{A}$. We have $V^*IdV{|_\mathcal{S}}=Id{|_\mathcal{S}}$, but
$$V^*Id(VV^*)V=I\neq VV^*=Id(VV^*).$$
This implies that $Id$ representation restricted to $\mathcal{S}$ has two UCP map extensions $V^*IdV$ and $Id$. Therefore the nondegenerate representation $Id{|_\mathcal{S}}$ does not have unique extension property. Using the Theorem \ref{hypuep}, $\mathcal{S}$ is not hyperrigid operator system in a $C^*$-algebra $\mathcal{A}$. This will imply that $\mathcal{G}_0$ is not hyperrigid in $\mathcal{A}$. 
\end{proof}

{\bf{Acknowledgment}:}
The author would like to thank Orr Moshe Shalit for valuable discussions and for a careful reading of this manuscript and some constructive comments. The author would like to thank Douglas Farenick and B. V. Rajarama Bhat for valuable discussions. The author would like to thank Statistics and Mathematics Unit, Indian Statistical Institute, Bangalore, India for providing visiting scientist post doctoral fellowship.

\bibliographystyle{amsplain}

\end{document}